\documentclass[3p, nopreprintline]{elsarticle}
\usepackage[T1]{fontenc}
\usepackage{amsfonts,amsmath,amsthm,amscd,amssymb,latexsym,amsbsy,caption}
\usepackage{bm} 
\usepackage{tikz}
\usetikzlibrary{positioning, calc, trees, decorations.markings, patterns, arrows, arrows.meta}
\usepackage{diagmac2}

\biboptions{sort&compress}

\newtheorem{theorem}{Theorem}[section]
\newtheorem{corollary}[theorem]{Corollary}
\newtheorem{lemma}[theorem]{Lemma}
\newtheorem{proposition}[theorem]{Proposition}
\newtheorem{conjecture}[theorem]{Conjecture}

\theoremstyle{definition}
\newtheorem{definition}[theorem]{Definition}
\newtheorem{example}[theorem]{Example}

\theoremstyle{remark}
\newtheorem{remark}[theorem]{Remark}

\numberwithin{equation}{section}

\makeatletter

\renewcommand{\p@enumii}{}
\makeatother

\newcommand{\RR}{\mathbb{R}}
\newcommand{\QQ}{\mathbb{Q}}
\newcommand{\NN}{\mathbb{N}}
\newcommand{\bB}{\mathbf{B}}

\DeclareMathOperator{\diam}{diam}
\def\<#1>{\langle #1 \rangle}

\newcommand{\myeq}{\stackrel{r}{=}}
\newcommand{\Lim}{\lim\limits}

\date{}
\journal{}

\begin{document}

\begin{frontmatter}
\title{The range of ultrametrics, compactness, and separability}

\author[1]{Oleksiy Dovgoshey\corref{cor1}}
\ead{oleksiy.dovgoshey@gmail.com}
\cortext[cor1]{Corresponding author}

\author[2]{Volodymir Shcherbak}
\ead{scherbakvf@ukr.net}

\address[1]{Department of Theory of Functions, Institute of Applied Mathematics and Mechanics of NASU, Dobrovolskogo str. 1, Slovyansk 84100, Ukraine}

\address[2]{Department of Applied Mechanics, Institute of Applied Mathematics and Mechanics of NASU, Dobrovolskogo str. 1, Slovyansk 84100, Ukraine}

\begin{abstract}
We describe the order type of range sets of compact ultrametrics and show that an ultrametrizable infinite topological space \((X, \tau)\) is compact iff the range sets are order isomorphic for any two ultrametrics compatible with the topology \(\tau\). It is also shown that an ultrametrizable topology is separable iff every compatible with this topology ultrametric has at most countable range set.
\end{abstract}


\end{frontmatter}

\section{Introduction}

In what follows we write \(\RR^{+}\) for the set of all nonnegative real numbers, \(\QQ\) for the set of all rational number, and \(\NN\) for the set of all strictly positive integer numbers.

\begin{definition}\label{d1.1}
A \textit{metric} on a set \(X\) is a function \(d\colon X\times X\rightarrow \RR^{+}\) satisfying the following conditions for all \(x\), \(y\), \(z \in X\):
\begin{enumerate}
\item \((d(x,y) = 0) \Leftrightarrow (x=y)\);
\item \(d(x,y)=d(y,x)\);
\item \(d(x, y)\leq \bigl(d(x, z) + d(z, y)\bigr)\).
\end{enumerate}
A metric space is a pair \((X, d)\) of a set \(X\) and a metric \(d\colon X\times X\rightarrow \RR^{+}\). A metric \(d\colon X\times X\rightarrow \RR^{+}\) is an \emph{ultrametric} on \(X\) if we have
\begin{equation}\label{d1.1:e3}
d(x,y) \leq \max \{d(x,z),d(z,y)\}
\end{equation}
for all \(x\), \(y\), \(z \in X\). Inequality~\eqref{d1.1:e3} is often called the \emph{strong triangle inequality}. 
\end{definition}

For every ultrametric space \(X\), each triangle in \(X\) is isosceles with the base being no greater than the legs. The converse statement also is valid: ``If \(Y\) is a metric space and each triangle in \(Y\) is isosceles with the base no greater than the legs, then \(Y\) is an ultrametric space.''.

The ultrametric spaces are connected with various of investigations in mathematics, physics, linguistics, psychology and computer science. Some properties of ultrametrics have been studied in~\cite{DM2009, DD2010, DP2013SM, GroPAMS1956, Lemin1984RMS39:5, Lemin1984RMS39:1, Lemin1985SMD32:3, Lemin1988, Lemin2003, Qiu2009pNUAA, Qiu2014pNUAA, BS2017, DM2008, DLPS2008TaiA, KS2012, Vau1999TP, Ves1994UMJ, Ibragimov2012, GH1961S, PTAbAppAn2014, Dov2019pNUAA, DP2020pNUAA, DovBBMSSS2020, VauAMM1975, VauTP2003}. The use of trees and tree-like structures gives a natural language for description of ultrametric spaces \cite{Carlsson2010, DLW, Fie, GV2012DAM, HolAMM2001, H04, BH2, Lemin2003, Bestvina2002, DDP2011pNUAA, DP2019PNUAA, DPT2017FPTA, PD2014JMS, DPT2015, Pet2018pNUAA, DP2018pNUAA, DKa2021, Dov2020TaAoG}.

Let \((X, d)\) be a metric space. An \emph{open ball} with a \emph{radius} \(r > 0\) and a \emph{center} \(c \in X\) is the set 
\[
B_r(c) = \{x \in X \colon d(c, x) < r\}.
\]
Write \(\bB_X = \bB_{X, d}\) for the set of all open balls in \((X, d)\).

\begin{definition}\label{d1.3}
Let \(\tau\) and \(d\) be a topology and, respectively, an ultrametric on a set \(X\). Then \(\tau\) and \(d\) are said to be compatible if \(\bB_{X, d}\) is an open base for the topology \(\tau\).
\end{definition}

Definition~\ref{d1.3} means that \(\tau\) and \(d\) are compatible if and only if every \(B \in \bB_{X, d}\) belongs to \(\tau\) and every nonempty \(A \in \tau\) can be written as the union of some family of elements of \(\bB_{X, d}\). If \((X, \tau)\) admits a compatible with \(\tau\) ultrametric on \(X\), then we say that the topological space \((X, \tau)\) (the topology \(\tau\)) is \emph{ultrametrizable}. Necessary and sufficient conditions under which topological spaces are ultrametrizable were found by De Groot \cite{GroPAMS1956, GroCM1958}. See also \cite{CLTaiA2020, BMTA2015, KSBLMS2012, BriTP2015, CSJPAA2019} for some future results connected with ultrametrizable topologies.

We define the \emph{distance set} \(D(X)\) of an ultrametric space \((X,d)\) as the range of the ultrametric \(d\colon X\times X\rightarrow \RR^{+}\),  
\[
D(X) = D(X, d) = \{d(x, y) \colon x, y \in X\}.
\]

In the present paper we focus on interconnections between ultrametrizable topologies and related properties of distance sets of ultrametrics which are compatible with these topologies.

The paper is organized as follows. 

Section~\ref{sec2} contains all necessary basic facts about metric spaces and some concepts related to totally ordered sets. 

Theorem~\ref{t5.10} of Section~\ref{sec3} describes the distance sets of totally bounded ultrametric spaces. In particular, it is shown that the distance sets are order isomorphic for all infinite totally bounded ultrametric spaces.

The main results of the paper are Theorems~\ref{t4.5} and \ref{t4.8} from Section~\ref{sec4}. Theorem~\ref{t4.5} claims that an ultrametrizable infinite topological space \((X, \tau)\) is compact iff any two compatible with \(\tau\) ultrametrics have the order isomorphic distance sets. In Theorem~\ref{t4.8} we prove that an ultrametrizable topological space \((X, \tau)\) is separable iff every compatible with \(\tau\) ultrametric has at most countable distance set.

\section{Basic definitions and facts}
\label{sec2}

\subsection{Metrics and ultrametrics}
\label{sec2.1}

Let \((X, d)\) be a metric space. A sequence 
\[
(x_n)_{n \in \mathbb{N}} \subseteq X
\]
is a \emph{Cauchy sequence} in \((X, d)\) if, for every \(r > 0\), there are \(x \in X\) and an integer \(n_0 \in \NN\) such that \(x_n \in B_r(x) \in \bB_{X}\) for every \(n \geqslant n_0\), i.e., \((x_n)_{n \in \mathbb{N}}\) eventually belongs to \(B_r(x)\).

\begin{remark}\label{r2.17}
Here and later the symbol \((x_n)_{n\in \mathbb{N}} \subseteq X\) means that \(x_n \in X\) holds for every \(n \in \mathbb{N}\).
\end{remark}

\begin{definition}\label{d2.2}
A sequence \((x_n)_{n \in \mathbb{N}}\) of points in a metric space \((X, d)\) is said to \emph{converge} to a point \(a \in X\),
\begin{equation*}
\lim_{n \to \infty} x_n = a,
\end{equation*}
if \((x_n)_{n \in \mathbb{N}}\) eventually belongs to every open ball \(B\) containing \(a\). 
\end{definition}

It is clear that every convergent sequence is a Cauchy sequence.

Now we recall a definition of \emph{total boundedness}.

\begin{definition}\label{d2.10}
A subset \(A\) of a metric space \((X, d)\) is totally bounded if for every \(r > 0\) there is a finite set \(\{B_r(x_1), \ldots, B_r(x_n)\} \subseteq \bB_X\) such that
\[
A \subseteq \bigcup_{i = 1}^{n} B_r(x_i).
\] 
\end{definition}

The next basic concept is the concept of \emph{completeness}.

\begin{definition}\label{d2.6}
A subset \(A\) of metric space \((X, d)\) is \emph{complete} if, for every Cauchy sequence \((x_n)_{n \in \mathbb{N}} \subseteq A\), there is a point \(a \in A\) such that 
\[
\lim_{n \to \infty} x_n = a.
\]
\end{definition}

An important subclass of complete metric spaces is the class of \emph{compact} metric spaces.

\begin{definition}[Borel---Lebesgue property]\label{d2.3}
Let \((X, d)\) be a metric space. A subset \(A\) of \(X\) is compact if every family \(\mathcal{F} \subseteq \bB_X\) satisfying the inclusion
\[
A \subseteq \bigcup_{B \in \mathcal{F}} B
\]
contains a finite subfamily \(\mathcal{F}_0 \subseteq \mathcal{F}\) such that 
\[
A \subseteq \bigcup_{B \in \mathcal{F}_0} B.
\]
\end{definition}

A standard definition of compactness of a topological space \((X, \tau)\) usually formulated as: Every open cover of \(X\) has a finite subcover.

A proof of the following well-known result can be found, for example, in~\cite[Theorem~12.1.3]{Sea2007}.

\begin{proposition}\label{p2.9}
Let \((X, d)\) be a metric space. Then the following conditions are equivalent for every \(A \subseteq X\):
\begin{enumerate}
\item \label{p2.9:s1} \(A\) is compact.
\item \label{p2.9:s2} Every infinite sequence of points of \(A\) contains a subsequence which converges to a point of \(A\).
\item \label{p2.9:s3} \(A\) is complete and totally bounded.
\end{enumerate}
\end{proposition}

Let \((X, \tau)\) be a topological space and let \(S\) be a subset of \(X\). If \(S \cap A \neq \varnothing\) holds for every nonempty \(A \in \tau\), then \(S\) is called a \emph{dense} subset of \(X\). For a metric space \((X, d)\) a set \(A \subseteq X\) is dense if and only if for every \(a \in X\) there is a sequence \((a_n)_{n\in \mathbb{N}} \subseteq A\) such that
\[
a = \lim_{n \to \infty} a_n.
\]
As usual we will say that a topological (or metric) space \(X\) is \emph{separable} iff \(X\) contains an at most countable dense subset. In particular, every finite topological space is separable.

Recall that a point \(p\) of a metric space \((X, d)\) is \emph{isolated} if there is \(\varepsilon > 0\) such that \(d(p, x) > \varepsilon\) for every \(x \in X \setminus \{p\}\). If \(p\) is not an isolated point of \(X\), then \(p\) is called an \emph{accumulation point} of \(X\). We say that \((X, d)\) is \emph{discrete} if all points of \(X\) are isolated.

It can be proved that for a given metric space \((X, d)\) there exists exactly one (up to isometry) complete metric space \((\widetilde{X}, \widetilde{d})\) such that \(\widetilde{X}\) contains a dense, isometric to \((X, d)\), subset (see, for example, Theorem~4.3.19 \cite{Eng1989}). The metric space \((\widetilde{X}, \widetilde{d})\) satisfying the above condition is called the \emph{completion of} \((X, d)\).

\begin{proposition}\label{p2.13}
The completion \((\widetilde{X}, \widetilde{d})\) of a metric space \((X, d)\) is compact if and only if \((X, d)\) is totally bounded.
\end{proposition}

For the proof see, for example, Corollary~4.3.30 in \cite{Eng1989}.

\begin{lemma}\label{l2.20}
The completion \((\widetilde{X}, \widetilde{d})\) is ultrametric for every ultrametric \((X, d)\).
\end{lemma}

\begin{proof}
Since \(\widetilde{d}\) is a metric on \(\widetilde{X}\), to prove that \((\widetilde{X}, \widetilde{d})\) is ultrametric it suffices to justify
\begin{equation}\label{l2.20:e2}
\widetilde{d} (x, y) \leqslant \max\{\widetilde{d} (x, z), \widetilde{d} (z, y)\}
\end{equation}
for all \(x\), \(y\), \(z \in \widetilde{X}\). By definition, the metric space \((\widetilde{X}, \widetilde{d})\) contains a dense subspace \((X', \widetilde{d}|_{X' \times X'})\) which is isometric to \((X, d)\). For simplicity we may assume \(X = X'\) and \(\widetilde{d}|_{X' \times X'} = d\).

Let \((x_n)_{n \in \mathbb{N}}\), \((y_n)_{n \in \mathbb{N}}\) and \((z_n)_{n \in \mathbb{N}}\) be sequences of points of \(X\) such that 
\begin{equation}\label{l2.20:e3}
\lim_{n \to \infty} x_n = x, \quad \lim_{n \to \infty} y_n = y \quad \text{and} \quad \lim_{n \to \infty} z_n = z.
\end{equation}
Since \(\widetilde{d} \colon \widetilde{X} \times \widetilde{X} \to \RR^{+}\) is a continuous mapping and \(d \colon X \times X \to \RR^{+}\) is an ultrametric, \eqref{l2.20:e3} implies
\begin{align*}
\widetilde{d} (x, y) = \lim_{n \to \infty} d(x_n, y_n) &\leqslant \max\left\{\lim_{n \to \infty} d(x_n, z_n), \lim_{n \to \infty} d(z_n, x_n)\right\} \\
&= \max\{\widetilde{d} (x, z), \widetilde{d} (z, y)\}.
\end{align*}
Inequality~\eqref{l2.20:e2} follows.
\end{proof}

The following lemma will be useful for description of distance sets of totally bounded ultrametric spaces.

\begin{lemma}\label{l6.4}
For every compact metric space \((X, d)\) the distance set \(D(X)\) is a compact subset of \((\RR^{+}, \rho)\) with respect to the standard Euclidean metric \(\rho\).
\end{lemma}

\begin{proof}
Let \((X, d)\) be compact. The metric \(d \colon X \times X \to \RR^{+}\) is a continuous function on the Cartesian product \(X \times X\). Any Cartesian product of compact spaces is compact by Tychonoff theorem (\cite[Theorem~3.2.4]{Eng1989}). Furthermore, every continuous image of compact Hausdorff space is compact (\cite[Theorem~3.1.10]{Eng1989}). To complete the proof it suffices to note that \((X, d)\) evidently is a Hausdorff space and that Cartesian products of Hausdorff spaces are Hausdorff (\cite[Theorem~2.3.11]{Eng1989}).
\end{proof}

The next result can be found in~\cite{Comicheo2018} (see formula (12), Theorem~1.6).

\begin{proposition}\label{c2.15}
Let \((X, d)\) be a nonempty ultrametric space and let \(A\) be a dense subset of \(X\). Then the distance sets of \((X, d)\) and of \((A, d|_{A \times A})\) are the same.
\end{proposition}

\begin{corollary}\label{c2.10}
Let \((X, \tau)\) be an ultrametrizable separable topological space. Then the distance set \(D(X, d)\) is at most countable for every compatible with~\(\tau\) ultrametric~\(d\).
\end{corollary}

\begin{remark}\label{r2.11}
In Theorem~\ref{t4.8} of Section~\ref{sec4} of the paper we will prove that the converse to Corollary~\ref{c2.10} is also true.
\end{remark}

Proposition~\ref{c2.15} and Lemma~\ref{l2.20} also imply the following well-known result: ``No new values of the ultrametric after completion'' (see, for example, \cite[p.~4]{PerezGarcia2010} or \cite[Theorem~1.6, Statement~(12)]{Comicheo2018}).

\begin{proposition}\label{p4.4}
Let \((X, d)\) be an ultrametric space and let \((\widetilde{X}, \widetilde{d})\) be the completion of \((X, d)\). Then the equality \(D(X) = D(\widetilde{X})\) holds.
\end{proposition}

\subsection{Totally ordered sets}
\label{sec2.2}

Recall that a reflexive and transitive binary relation \(\preccurlyeq\) on a set \(Y\) is a \emph{partial order} on \(Y\) if \(\preccurlyeq\) has the \emph{antisymmetric property}:
\[
\bigl(x \preccurlyeq y \text{ and } y \preccurlyeq x\bigr) \Rightarrow (x = y)
\]
is valid for all \(x\), \(y \in Y\). 

\begin{remark}\label{r3.1}
Here and in what follows we use \(x \preccurlyeq y\) instead of \(\<x, y> \in {\preccurlyeq}\).
\end{remark}

Let \(\preccurlyeq_Y\) be a partial order on a set \(Y\). Then a pair \((Y, {\preccurlyeq}_Y)\) is called to be a \emph{poset} (a partially ordered set). If \(Z\) is a subset of \(Y\) and \({\preccurlyeq}_Z\) is a partial order on \(Z\) given by \({\preccurlyeq}_Z = Z^{2} \cap {\preccurlyeq}_Y\), then we say that \(Z = (Z, {\preccurlyeq}_Z)\) is a subposet of \((Y, {\preccurlyeq}_Y)\). A poset \((Y, {\preccurlyeq}_Y)\) is \emph{totally ordered} if, for all points \(y_1\), \(y_2 \in Y\), we have 
\[
y_1 \preccurlyeq_Y y_2 \quad \text{or} \quad y_2 \preccurlyeq_Y y_1.
\]

\begin{example}\label{ex3.2}
Let \(\mathbb{Z}^{-}\) be the set of all negative integers,
\[
\mathbb{Z}^{-} = \{\ldots, -3, -2, -1\}
\]
and let \(\widehat{\mathbb{Z}}^{-} := \mathbb{Z}^{-} \cup \{-\infty\}\). Let us define an order \(\preccurlyeq\) on \(\widehat{\mathbb{Z}}^{-}\) by:
\[
-\infty \preccurlyeq p \quad \text{and} \quad \bigl((q \preccurlyeq l) \Leftrightarrow (q \leqslant l)\bigr)
\]
for all \(p\), \(q\), \(l \in \mathbb{Z}^{-}\), where \(\leqslant\) is the standard order on \(\mathbb{R}\). Then \((\widehat{\mathbb{Z}}^{-}, {\preccurlyeq})\) is totally ordered.
\end{example}

\begin{definition}\label{d2.8}
Let \((Q, {\preccurlyeq}_Q)\) and \((L, {\preccurlyeq}_L)\) be posets. A mapping \(f \colon Q \to L\) is \emph{isotone} if, for all \(q_1\), \(q_2 \in Q\), we have
\[
(q_1 \preccurlyeq_Q q_2) \Rightarrow (f(q_1) \preccurlyeq_L f(q_2)).
\]
If, in addition, the isotone mapping \(f \colon Q \to L\) is bijective and the inverse mapping \(f^{-1} \colon L \to Q\) is also isotone, then we say that \(f\) is an \emph{order isomorphism} and that \((Q, {\preccurlyeq}_Q)\) and \((L, {\preccurlyeq}_L)\) have the same \emph{order type}.
\end{definition}

\begin{definition}\label{d6.3}
A poset has the order type \(\mathbf{1} + \bm{\omega}^{*}\) if this poset is order isomorphic to \((\widehat{\mathbb{Z}}^{-}, {\preccurlyeq})\).
\end{definition}

\begin{example}\label{ex3.5}
The subposets 
\[
\{0\} \cup \left\{\frac{1}{n} \colon n \in \mathbb{N}\right\} \quad \text{and} \quad \{0\} \cup \left\{1+\frac{1}{n} \colon n \in \mathbb{N}\right\}
\]
of the poset \((\mathbb{R}, {\leqslant})\) have the order type \(\mathbf{1} + \bm{\omega}^{*}\).
\end{example}

\begin{remark}\label{r6.2}
The symbol \(\bm{\omega}\) is the usual designation of the order type of \((\mathbb{N}, {\leqslant})\) and \(\mathbf{1}\) denotes the order type of one-point posets. The symbol \(\bm{\omega}^{*}\) is used for the backwards order type of \(\bm{\omega}\). Thus, the subposet \(\mathbb{Z}^{-}\) of the poset \((\mathbb{R}, {\leqslant})\) has the order type \(\bm{\omega}^{*}\). Moreover, by \(\mathbf{1} + \bm{\omega}^{*}\) we denote the sum of the order types \(\mathbf{1}\) and \(\bm{\omega}^{*}\) (see, for example, Section~1.4 in \cite{Ros1982} for details).
\end{remark}

We also remember that a point \(q_0 \in Q\) is the called the smallest element of a poset \((Q, {\preccurlyeq})\) if the inequality \(q_0 \preccurlyeq q\) holds for every \(q \in Q\). The largest element of the poset can be defined by duality. In particular, every totally ordered set having the order type \(\mathbf{1} + \bm{\omega}^{*}\) contains the largest element and the smallest one.

\begin{example}\label{ex2.20}
The point \(0\) is the smallest element of the distance set \(D(X, d)\) of every nonempty ultrametric space \((X, d)\). The distance set \(D(X, d)\) contains the largest element iff \(\diam X \in D(X, d)\), where \(\diam (X, d) = \sup\{d(x, y) \colon x, y \in X\}\).
\end{example}

\section{Distance sets of totally bounded ultrametric spaces}
\label{sec3}

The following general result was obtained by Delhomm\'{e}, Laflamme, Pouzet, and Sauer in paper~\cite{DLPS2008TaiA}.

\begin{theorem}\label{t5.6}
Let \(A\) be a subset of \(\RR^{+}\). Then \(A\) is the distance set of a nonempty ultrametric space if and only if \(0\) belongs to \(A\).
\end{theorem}

The proof of this theorem is based on the following simple construction.

\begin{example}\label{ex3.11}
Let \(X\) be a subset of \(\RR^{+}\) such that \(0 \in X\). Following \cite{DLPS2008TaiA} we define a mapping \(d \colon X \times X \to \RR^{+}\) by
\begin{equation}\label{ex3.11:e1}
d(x, y) = \begin{cases}
0 & \text{if } x = y,\\
\max \{x, y\} & \text{if } x \neq y.
\end{cases}
\end{equation}
Then \(d\) is an ultrametric on \(X\) and the distance set of \((X, d)\) coincides with \(X\), \(D(X, d) = X\).
\end{example}

In Theorem~\ref{t5.10} below we will refine Theorem~\ref{t5.6} for the case of totally bounded ultrametric spaces.

\begin{lemma}\label{l6.3}
Let \((X, d)\) be a totally bounded infinite ultrametric space. Then \(D(X)\) is countably infinite, \(|D(X)| = \aleph_0\).
\end{lemma}

\begin{proof}
Every totally bounded metric space is separable. Consequently, \(D(X)\) is at most countable by Corollary~\ref{c2.10}.

Suppose \(D(X)\) is finite, \(|D(X)| < \aleph_0\). Since \(X\) is infinite, the inequality \(|D(X)| \geqslant 2\) holds. The double inequality \(2 \leqslant |D(X)| < \aleph_0\) implies that there is \(d_1 \in D(X)\) such that
\begin{equation}\label{l6.3:e1}
0 < d_1 \leqslant d^*
\end{equation}
for every \(d^* \in D(X) \setminus \{0\}\). Let us consider the family \(\mathcal{F} = \left\{B_{d_1/2}(x) \colon x \in X\right\}\). Then \(\mathcal{F}\) is a cover of \(X\), 
\[
X \subseteq \bigcup_{B \in \mathcal{F}} B.
\]
Using~\eqref{l6.3:e1} we obtain \(B_{d_1/2}(x) \cap B_{d_1/2}(y) = \varnothing\) for all distinct \(x\), \(y \in X\). Hence, the cover \(\mathcal{F}\) does not contain any finite subcover. This contradicts Definition~\ref{d2.10}. The equality \(|D(X)| = \aleph_0\) follows.
\end{proof}

The following lemma is a special case of Lemma~1 from \cite{DP2013SM}.

\begin{lemma}\label{l3.10}
Let \((X, \rho)\) be an ultrametric space and let \(\{a, b, c, d\}\) be a four-point subset of \(X\). Then the inequalities \(\rho(a, b) > \rho(a, d)\) and \(\rho(a, b) > \rho(b, c)\) imply the equality \(\rho(a, b) = \rho(c, d)\).
\end{lemma}

\begin{lemma}\label{l6.5}
Let \((X, d)\) be a compact ultrametric space with the distance set \(D(X)\) and let \(p > 0\) belong to \(D(X)\). Then \(p\) is an isolated point of the metric space \((D(X), \rho)\) with the standard metric \(\rho(x, y) = |x-y|\).
\end{lemma}

\begin{proof}
Suppose for every \(\varepsilon > 0\) there is \(s \in D(X)\) satisfying
\begin{equation}\label{l6.5:e1}
p < s < p+\varepsilon.
\end{equation}
Then there is \((s_n)_{n \in \mathbb{N}} \subseteq D(X)\) such that 
\begin{equation}\label{l6.5:e7}
\lim_{n \to \infty} s_n = p \quad \text{and} \quad p < s_n
\end{equation}
for every \(n \in \mathbb{N}\). We can find two sequences \((x_n)_{n \in \mathbb{N}} \subseteq X\) and \((y_n)_{n \in \mathbb{N}} \subseteq X\) such that
\begin{equation}\label{l6.5:e2}
d(x_n, y_n) = s_n
\end{equation}
for every \(n \in \mathbb{N}\) and 
\begin{equation}\label{l6.5:e3}
\lim_{n \to \infty} d(x_n, y_n) = p.
\end{equation}
By Proposition~\ref{p2.9}, there are points \(x\), \(y \in X\) and subsequences \((x_{n_k})_{k \in \mathbb{N}}\) of \((x_n)_{n \in \mathbb{N}}\) and \((y_{n_k})_{k \in \mathbb{N}}\) of \((y_n)_{n \in \mathbb{N}}\) such that 
\begin{equation}\label{l6.5:e4}
\lim_{k \to \infty} d(x, x_{n_k}) = \lim_{k \to \infty} d(y, y_{n_k}) = 0.
\end{equation}
Since \(d \colon X \times X \to \RR^{+}\) is continuous, from~\eqref{l6.5:e7}, \eqref{l6.5:e2} and \eqref{l6.5:e4} it follows that \(d(x, y) = p\). 

Let us consider the set \(\{x,  y, x_{n_k}, y_{n_k}\} \subseteq X\). Equalities \(d(x, y) = p\), \eqref{l6.5:e4} and Lemma~\ref{l3.10} give us \(p = d(x_{n_k}, y_{n_k})\) for all sufficiently large \(k \in \mathbb{N}\), that contradicts \eqref{l6.5:e7}. Thus, there is \(\varepsilon_1 > 0\) such that
\begin{equation}\label{l6.5:e5}
(p, p+\varepsilon_1) \cap D(X) = \varnothing.
\end{equation}

Arguing similarly we can find \(\varepsilon_2 > 0\) satisfying 
\begin{equation}\label{l6.5:e6}
(p - \varepsilon_2, p) \cap D(X) = \varnothing.
\end{equation}
From~\eqref{l6.5:e5} and \eqref{l6.5:e6} it follows that \(p\) is an isolated point of \(D(X)\).
\end{proof}

\begin{theorem}\label{t5.10}
Let \(A\) be a subset of \(\RR^{+}\). Then \(A\) is the distance set of an infinite totally bounded ultrametric space \((X, d)\) if and only if the following conditions simultaneously hold:
\begin{enumerate}
\item \label{t5.10:s2} \(A\) has the order type \(\mathbf{1} + \bm{\omega}^{*}\) as a subposet of \((\RR^{+}, {\leqslant})\).
\item \label{t5.10:s3} The point \(0\) is an accumulation point of \(A\) with respect to the standard Euclidean metric on \(\RR^{+}\).
\end{enumerate}
\end{theorem}

\begin{proof}
Suppose that there is an infinite totally bounded ultrametric space \((X, d)\) such that \(A = D(X)\). Let us prove the validity of \ref{t5.10:s2}. The completion \(\widetilde{X}\) of \(X\) is compact (see Proposition~\ref{p2.13}) and ultrametric (see Lemma~\ref{l2.20}). By Corollary~\ref{p4.4}, the equality \(D(X) = D(\widetilde{X})\) holds. Hence, without loss generality, we assume that \(X\) is compact. 

By Lemma~\ref{l6.3}, we have \(|D(X)| = \aleph_0\). Consequently, it suffices to show that the cardinality of 
\[
D_{d^*} := D(X) \cap [d^{*}, \diam X]
\]
is finite for every strictly positive \(d^* \in D(X)\).

By Lemma~\ref{l6.4}, the distance set \(D(X)\) is a compact subset of \(\RR\). The classical Heine---Borel theorem (from Real Analysis) claims that a subset \(S\) of \(\RR\) is compact if and only if \(S\) is closed and bounded. Thus, \(D_{d^*}\) is compact. By Lemma~\ref{l6.5}, every \(p \in D_{d^*}\) is isolated. Hence, for every \(p \in D_{d^*}\), the one-point set \(\{p\}\) is an open ball in the metric subspace \(D_{d^*}\) of \(\RR\). It is clear that the open cover 
\[
\bigl\{\{p\} \colon p \in D_{d^*}\bigr\}
\]
contains a finite subcover if and only if the cardinal number \(|D_{d^*}|\) is finite. Now the inequality \(|D_{d^*}| < \infty\) follows from Borel---Lebesgue property (Definition~\ref{d2.3}). Condition~\ref{t5.10:s2} follows.

Let us prove condition \ref{t5.10:s3}. Using Proposition~\ref{p4.4} and Lemma~\ref{l6.4} we obtain that \(A\) is a compact subset of \(\RR^{+}\). If \ref{t5.10:s3} does not hold, then \(0\) is an isolated point of \(A\). Consequently, by Lemma~\ref{l6.5}, all points of \(A\) are isolated. Since \(A\) is compact, it implies the finiteness of \(A\), contrary to \ref{t5.10:s2}. Condition~\ref{t5.10:s3} follows.

Conversely, suppose that \(A\) satisfies conditions \ref{t5.10:s2} and \ref{t5.10:s3}. First of all we note that condition~\ref{t5.10:s3} implies \(0 \in A\). Consequently, if we write \(X = A\) and define the ultrametric \(d \colon X \times X \to \RR^{+}\) as in Example~\ref{ex3.11}. Then the equality~\(A = D(X)\) holds. Using \ref{t5.10:s2} and \ref{t5.10:s3}, it is easy to prove that \((X, d)\) is compact, ultrametric and infinite.
\end{proof}

\begin{remark}\label{r6.11}
Considering the poset \((\RR^{+}, {\leqslant})\) and its subposet
\[
\{0\} \cup \left\{1+\frac{1}{n} \colon n \in \mathbb{N}\right\}
\]
(Example~\ref{ex3.5}), we see that conditions \ref{t5.10:s2} and \ref{t5.10:s3} from Theorem~\ref{t5.10} are independent of one another.
\end{remark}

\begin{example}\label{ex6.8}
Let \((x_{n})_{n \in \mathbb{N}} \subseteq \RR^{+}\) be a strictly decreasing sequence such that
\[
\lim_{n \to \infty} x_n = 0.
\]
Let us define a set \(A \subseteq \RR^{+}\) by the rule
\[
(x \in A) \Leftrightarrow (x=0 \text{ or } \exists n \in \mathbb{N} \colon x = x_n).
\]
Write \(X = A\) and let \(d \colon X \times X \to \RR^{+}\) be defined as in Example~\ref{ex3.11}. Then \((X, d)\) is a compact infinite ultrametric space satisfying the equality \(D(X, d) = A\).
\end{example}

Theorem~\ref{t5.10} and Example~\ref{ex6.8} give us the following ``constructive'' description of distance sets of all infinite totally bounded ultrametric spaces.

\begin{corollary}\label{c6.9}
The following statements are equivalent for every \(A \subseteq \RR^{+}\):
\begin{enumerate}
\item \label{c6.9:s1} There is an infinite compact ultrametric space \(X\) such that \(A = D(X)\).
\item \label{c6.9:s3} There is an infinite totally bounded ultrametric space \(X\) such that \(A = D(X)\).
\item \label{c6.9:s2} There is a strictly decreasing sequence \((x_n)_{n \in \mathbb{N}} \subseteq \RR^{+}\) such that 
\[
\lim_{n \to \infty} x_n = 0
\]
holds and the equivalence 
\[
(x \in A) \Leftrightarrow (x = 0 \text{ or } \exists n \in \mathbb{N} \colon x_n = x)
\]
is valid for every \(x \in \RR^{+}\).
\end{enumerate}
\end{corollary}

\section{From the distance sets to compactness and separability}
\label{sec4}

Let \((X, \tau)\) be a topological space. Recall that a metric \(d \colon X \times X \to \RR^{+}\) is compatible with \(\tau\) if the equality \(\tau = \tau_d\) holds, where we denote by \(\tau_d\) a topology with the base \(\bB_{X} = \bB_{X, d}\). The topological space \((X, \tau)\) is, by definition, ultrametrizable if there is a compatible with \(\tau\) ultrametric on \(X\).

Before considering the set of all ultrametrics, which are compatible with an compact ultrametrizable topology, we will formulate some lemmas.

The first lemma is a reformulation of an elegant theorem that was obtained by Pongsriiam and Termwuttipong \cite{PTAbAppAn2014}.

\begin{lemma}\label{l4.1}
The following conditions are equivalent for every \(f \colon \RR^{+} \to \RR^{+}\):
\begin{enumerate}
\item \label{l4.1:s1} The mapping \(X \times X \xrightarrow{d} \RR^{+} \xrightarrow{f} \RR^{+}\) is an ultrametric for every ultrametric space \((X, d)\).
\item \label{l4.1:s2} \(f\) is increasing and the equality \(f(x) = 0\) holds if and only if \(x = 0\).
\end{enumerate}
\end{lemma}

\begin{remark}\label{r4.3}
The functions satisfying condition~\ref{l4.1:s1} of Lemma~\ref{l4.1} are called ultrametric preserving. These functions together with the construction of ultrametrics given in \cite{DLPS2008TaiA} are played a critical role in the proofs of main results of the paper. Some future information on ultrametric preserving functions can be found in \cite{BDSa2020, SKPMS2020, KP2018MS, KTTJM2007, Dov2020MS, VD2019a, JTRRACEFNSAMR2020, KPS2019IJoMaCS}
\end{remark}

\begin{lemma}\label{l4.2}
Let \((X, \rho)\) be an infinite ultrametric space. Then there is an ultrametric \(d \colon X \times X \to \RR^{+}\) such that \(\tau_d = \tau_{\rho}\) and the distance set \(D(X, d)\) has the order type \(\mathbf{1} + \bm{\omega}^{*}\).
\end{lemma}

\begin{proof}
Suppose first that \((X, \rho)\) is discrete. Let \(\{X_n \colon n \in \NN\}\) be a partitions of \(X\), i.e., 
\begin{equation}\label{l4.2:e1}
X = \bigcup_{n \in \NN} X_n \quad \text{and} \quad X_{n_1} \neq \varnothing = X_{n_1} \cap X_{n_2}
\end{equation}
hold for all distinct \(n_1\), \(n_2 \in \NN\). Similarly to \eqref{ex3.11:e1} we define a mapping \(d \colon X \times X \to \RR^{+}\) as
\begin{equation}\label{l4.2:e2}
d(x, y) = \begin{cases}
0 & \text{if } x = y,\\
\max\left\{\frac{1}{n}, \frac{1}{m}\right\} & \text{if \(x \neq y\) and \(x \in X_n\), \(y \in X_m\)}.
\end{cases}
\end{equation}
It is clear that \(d\) is a symmetric nonnegative mapping and \(d(x, y) = 0\) iff \(x = y\). We claim that the strong triangle inequality
\begin{equation}\label{l4.2:e5}
d(x, y) \leqslant \max \{d(x, z), d(z, y)\}
\end{equation}
holds for all \(x\), \(y\), \(z \in X\). Indeed, it suffices to prove \eqref{l4.2:e5} when the points \(x\), \(y\) and \(z\) are pairwise distinct (in opposite case inequality \eqref{l4.2:e5} is evidently valid). Let \(x\), \(y\), \(z\) be pairwise distinct and \(x \in X_p\), \(y \in X_q\), \(z \in X_l\) for some not necessarily different \(p\), \(q\), \(l \in \NN\). Then using \eqref{l4.2:e2} we see that \eqref{l4.2:e5} is equivalent to the true inequality 
\begin{equation}\label{l4.2:e6}
\max \left\{\frac{1}{p}, \frac{1}{q}\right\} \leqslant \max \left\{\max \left\{\frac{1}{p}, \frac{1}{l}\right\}, \max \left\{\frac{1}{l}, \frac{1}{q}\right\}\right\}.
\end{equation}
Hence, \(d\) is an ultrametric on \(X\). If \(x_0\) is an arbitrary point of \(X\), then there is \(n \in \NN\) such that \(x_0 \in X_n\). Now \eqref{l4.2:e2} implies that the inequality \(d(x_0, x) \geqslant \frac{1}{n}\) holds for every \(x \in X\). Hence, \((X, d)\) is a discrete ultrametric space, that implies \(\tau_d = \tau_{\rho}\). From~\eqref{l4.2:e1} and \eqref{l4.2:e2} follows the equality
\[
D(X, d) = \left\{\frac{1}{n} \colon n \in \NN\right\} \cup \{0\}.
\]
Thus, \(D(X, d)\) has the order type \(\mathbf{1} + \bm{\omega}^{*}\) (see Example~\ref{ex3.5}).

If \((X, \rho)\) contains an accumulation point \(p\), then there is a sequence \((x_n)_{n \in \NN} \subseteq X\) such that \(\lim_{n \to \infty} x_n = p\). Hence, 
\begin{equation}\label{l4.2:e3}
(\rho(x_n, p))_{n \in \NN} \subseteq D(X, d)
\end{equation}
and \(0\) is an accumulation point of the set \(D(X, d)\) with respect to the standard Euclidean metric on \(\RR^{+}\). Let us define a function \(f \colon \RR^{+} \to \RR^{+}\) as
\begin{equation}\label{l4.2:e4}
f(t) = \begin{cases}
0 & \text{if } t = 0,\\
1 & \text{if } t \geqslant 1,\\
\frac{1}{n-1} & \text{if } \frac{1}{n-1} < t \leqslant \frac{1}{n}, \quad n \in \NN.
\end{cases}
\end{equation}
Let us denote by \(d\) the mapping \(X \times X \xrightarrow{\rho} \RR^{+} \xrightarrow{f} \RR^{+}\). Since \(f\) satisfies condition~\ref{l4.1:s2} of Lemma~\ref{l4.1}, \(d\) is an ultrametric on \(X\). Using~\eqref{l4.2:e4} it is easy to see that, for every \((y_n)_{n \in \NN}\) and every \(z \in X\), we have \(\lim_{n \to \infty} d(y_n, z) = 0\) if and only if \(\lim_{n \to \infty} \rho(y_n, z) = 0\). It implies the equality \(\tau_d = \tau_{\rho}\). Moreover, from~\eqref{l4.2:e3} and definition of \(d\) it follows that \((0, \infty) \cap D(X, d)\) is an infinite subset of the set \(\left\{\frac{1}{n} \colon n \in \NN\right\}\). All infinite subposets of \(\left\{\frac{1}{n} \colon n \in \NN\right\}\) has the order type \(\bm{\omega}^{*}\). Hence, the poset \(D(X, d)\),
\[
D(X, d) = \{0\} \cup \bigl((0, \infty) \cap D(X, d)\bigr),
\]
has the order type \(\mathbf{1} + \bm{\omega}^{*}\).
\end{proof}

In the next lemma we present a well-known property of open ultrametric balls.

\begin{lemma}\label{l4.3}
Let \((X, d)\) be an ultrametric space. Then, for every \(B_r(c) \in \bB_{X}\) and every \(a \in B_r(c)\), we have \(B_r(c) = B_r(a)\).
\end{lemma}

For the proof of this lemma see, for example, Proposition~18.4 \cite{Sch1985}.

\begin{corollary}\label{c4.4}
Let \((X, d)\) be an ultrametric space. Then, for all \(B_{r_1}(c_1)\) and \(B_{r_2}(c_2)\), we have
\begin{equation}\label{c4.4:e1}
B_{r_1}(c_1) \subseteq B_{r_2}(c_2)
\end{equation}
whenever \(B_{r_1}(c_1) \cap B_{r_2}(c_2) \neq \varnothing\) and \(0 < r_1 \leqslant r_2 < \infty\). In particular, for all \(B_{r}(c_1)\) and \(B_{r}(c_2)\) and every \(r >0\) we have
\begin{equation}\label{c4.4:e2}
B_{r}(c_1) = B_{r}(c_2)
\end{equation}
whenever \(B_{r}(c_1) \cap B_{r}(c_2) \neq \varnothing\).
\end{corollary}

\begin{proof}
Let \(B_r(c_1) \cap B_r(c_2) \neq \varnothing\) hold and \(r > 0\) be fixed. Then, by Lemma~\ref{l4.3}, we obtain the equalities 
\[
B_r(c_1) = B_r(a)\quad \text{and} \quad B_r(c_2) = B_r(a)
\]
for every \(a \in B_r(c_1) \cap B_r(c_2)\). Equality~\eqref{c4.4:e2} follows.

If we have \(0 < r_1 \leqslant r_2 < \infty\) and \(B_{r_1}(c_1) \cap B_{r_2}(c_2) \neq \varnothing\), then \eqref{c4.4:e2} implies the equality \(B_{r_2}(c_1) = B_{r_2}(c_2)\). Now \eqref{c4.4:e1} follows from the last equality and the inclusion \(B_{r_1}(c_1) \subseteq B_{r_2}(c_1)\).
\end{proof}

\begin{lemma}\label{l4.6}
Let \((X, d)\) be an ultrametric space, let \(r > 0\) be fixed, and let \(C\) be a nonempty subset of \(X\). Then there is \(C^{1} \subseteq C\) such that
\begin{equation}\label{l4.6:e1}
\bigcup_{c \in C} B_r(c) = \bigcup_{c \in C^{1}} B_r(c)
\end{equation}
and \(B_r(c_1) \cap B_r(c_2) = \varnothing\) whenever \(c_1\) and \(c_2\) are distinct points of \(C^{1}\).
\end{lemma}

\begin{proof}
Let us define a binary relation \(\myeq\) on the set \(C\) as
\[
(c_1 \myeq c_2) \Leftrightarrow (B_r(c_1) \cap B_r(c_2) \neq \varnothing).
\]
It is clear that the relation \(\myeq\) is reflexive and symmetric,
\[
c \myeq c, \quad \text{and} \quad (a \myeq b) \Leftrightarrow (b \myeq a)
\]
for all \(a\), \(b\), \(c \in C\). Moreover, Corollary~\ref{c4.4} implies the the transitivity of \(\myeq\). Hence, \(\myeq\) is an equivalence relation on \(C\).

Using the well-known one-to-one correspondence between the equivalence relations and partitions (see, for example, \cite[Chapter~II, \S~5]{KurMost}), we can find a partition \(\widetilde{C} = \{C_i \colon i \in I\}\) of the set \(C\) generated by relation \(\myeq\). It means that, for all \(a\), \(b \in C\), we have \(a \myeq b\) iff there is \(i \in I\) such that \(a\), \(b \in C_i\) and
\begin{equation*}
\bigcup_{i \in I} C_i = C, \quad C_i \cap C_j = \varnothing \quad \text{and} \quad C_k \neq \varnothing
\end{equation*}
whenever \(i\), \(j\), \(k \in I\) and \(i \neq j\). Now if we define \(C^{1} \subseteq C\) as a system of distinct representatives for \(\widetilde{C}\),
\[
C^{1} = \{c_i \colon i \in I \text{ and } c_i \in C_i\},
\]
then, using Corollary~\ref{c4.4}, we obtain the equality
\begin{equation}\label{l4.6:e3}
B_r(c_i) = \bigcup_{c \in C_i} B_r(c)
\end{equation}
for every \(i \in I\). In addition, from the definition of \(\myeq\) it follows that \(B_r(c_{i_1})\) and \(B_r(c_{i_2})\) are disjoint for all distinct \(i_1\), \(i_2 \in I\). To complete the proof it suffices to note that \eqref{l4.6:e1} follows from the definition of \(\widetilde{C}\) and \eqref{l4.6:e3}.
\end{proof}

\begin{theorem}\label{t4.5}
Let \((X, \tau)\) be an ultrametrizable infinite topological space. Then the following conditions are equivalent:
\begin{enumerate}
\item\label{t4.5:s1} The space \((X, \tau)\) is compact.
\item\label{t4.5:s2} The distance set \(D(X, d)\) of an ultrametric space \((X, d)\) has the order type \(\mathbf{1} + \bm{\omega}^{*}\) whenever \(d\) is compatible with \(\tau\).
\item\label{t4.5:s3} If \((X, d)\) and \((X, \rho)\) are ultrametric spaces such that 
\[
\tau_d = \tau = \tau_{\rho},
\]
then the distance sets \(D(X, d)\) and \(D(X, \rho)\) are order isomorphic as subposets of \((\RR^{+}, {\leqslant})\).
\item\label{t4.5:s4} The distance set \(D(X, d)\) has the largest element whenever \(d\) is a compatible with \(\tau\) ultrametric.
\end{enumerate}
\end{theorem}

\begin{proof}
\(\ref{t4.5:s1} \Rightarrow \ref{t4.5:s2}\). This is valid by Theorem~\ref{t5.10}.

\(\ref{t4.5:s2} \Rightarrow \ref{t4.5:s3}\). If \ref{t4.5:s2} holds, then we have \ref{t4.5:s3} by Definition~\ref{d2.8}.

\(\ref{t4.5:s3} \Rightarrow \ref{t4.5:s4}\). By Lemma~\ref{l4.2}, there is an ultrametric \(\rho \colon X \times X \to \RR^{+}\) such that \(\tau = \tau_{\rho}\) and \(D(X, \rho)\) has the order type \(\mathbf{1} + \bm{\omega}^{*}\). If \ref{t4.5:s3} holds and an ultrametric \(d \colon X \times X \to \RR^{+}\) is compatible with \(\tau\), then \(D(X, d)\) also has the order type \(\mathbf{1} + \bm{\omega}^{*}\) by \ref{t4.5:s3} and, consequently, \(D(X, d)\) contains the largest element.

\(\ref{t4.5:s4} \Rightarrow \ref{t4.5:s1}\). Let \ref{t4.5:s4} hold. Suppose that \((X, \tau)\) is not a compact. Using Lemma~\ref{l4.2} we can find an ultrametric \(\rho \colon X \times X \to \RR^{+}\) such that \(\tau = \tau_{\rho}\) and \(\mathbf{1} + \bm{\omega}^{*}\) is the order type of \(D(X, \rho)\). Hence, \(D(X, \rho)\) has the largest element \(\diam (X, \rho)\). Thus, it suffices to find a compatible with \(\tau\) ultrametric \(d \colon X \times X \to \RR^{+}\) such that \(D(X, d)\) does not have the largest element.

We will construct a desirable ultrametric \(d \colon X \times X \to \RR^{+}\) by a modification of the ultrametric \(\rho\). Consider first the case when \((X, \rho)\) is not totally bounded. Then, by Definition~\ref{d2.10}, there exists \(r_1 > 0\) such that the open cover 
\[
\widetilde{\mathcal{B}} = \left\{B_{r_1}(x) \colon x \in X, \ B_{r_1}(x) \in \bB_{X, \rho}\right\}
\]
of \(X\) does not have any finite subcover. Using Lemma~\ref{l4.6}, we find a set \(C \subseteq X\) such that
\[
\bigcup_{c \in C} B_{r_1}(c) = X
\]
and \(B_{r_1}(c_1) \cap B_{r_1}(c_2) = \varnothing\) whenever \(c_1\) and \(c_2\) are distinct points of \(C\). Write \(\mathcal{B} = \{B_{r_1}(c) \colon c \in C\}\). Since the cover \(\widetilde{\mathcal{B}}\) has no finite subcovers, \(\mathcal{B}\) has an infinite cardinality. Consequently, there is a surjective mapping \(g \colon \mathcal{B} \to (r_1, 2r_1) \cap \QQ\). Let us define a mapping \(d \colon X \times X \to \RR^{+}\) as
\begin{equation}\label{t4.5:e1}
d(x, y) = \begin{cases}
\rho(x, y) & \parbox[t]{.45\textwidth}{if there is \(B \in \mathcal{B}\) such that \(x\), \(y \in B\),}\\
\max\{g(B^1), g(B^2)\} & \parbox[t]{.45\textwidth}{if \(x \in B^1\) and \(y \in B^2\) for distinct \(B^1\), \(B^2 \in \mathcal{B}\).}
\end{cases}
\end{equation}
We claim that \(d\) is an ultrametric on \(X\). Indeed, \(d\) is symmetric and \(d(x, y) = 0\) holds iff \(x = y\). Thus, it suffices to prove the strong triangle inequality
\begin{equation}\label{t4.5:e2}
d(x, y) \leqslant \max\left\{d(x, z), d(z, y)\right\}
\end{equation}
for all \(x\), \(y\), \(z \in X\). Let us consider all possible cases of mutual arrangement of the points \(x\), \(y\) and \(z\). If there is \(B \in \mathcal{B}\) such that \(\{x, y, z\} \subseteq B\), then from \eqref{t4.5:e1} it follows that 
\[
d(x, y) = \rho(x, y), \quad d(x, z) = \rho(x, z) \quad \text{and} \quad d(y, z) = \rho(y, z).
\]
Hence, \eqref{t4.5:e2} follows because \(\rho\) is an ultrametric. If there are distinct \(B^1\), \(B^2\), \(B^3 \in \mathcal{B}\) such that \(x \in B^1\), \(y \in B^2\) and \(z \in B^3\), then using \eqref{t4.5:e1} we can rewrite \eqref{t4.5:e2} as the obviously correct inequality
\[
\max\{g(B^1), g(B^2)\} \leqslant \max\bigl\{\max\{g(B^1), g(B^3)\}, \max\{g(B^3), g(B^2)\}\bigr\}
\]
(cf. inequality~\eqref{l4.2:e6} from the proof of Lemma~\ref{l4.2}). To complete the proof of \eqref{t4.5:e2} it suffices to consider the case when there exist some distinct \(B^1\), \(B^2 \in \mathcal{B}\) such that \(\{x, y, z\} \subseteq B^1 \cup B^2\) but 
\[
B \setminus \{x, y, z\} \neq \varnothing
\]
for every \(B \in \mathcal{B}\). Without loss of generality, we suppose that
\begin{equation}\label{t4.5:e4}
d(x, y) \geqslant d(x, z) \geqslant d(z, y) > 0
\end{equation}
and 
\begin{equation}\label{t4.5:e5}
|B^1 \cap \{x, y, z\}| = 2, \quad |B^2 \cap \{x, y, z\}| = 1.
\end{equation}
Using Lemma~\ref{l4.3} and \eqref{t4.5:e4}, \eqref{t4.5:e5} we can prove the inclusions
\[
\{y, z\} \subseteq B^1 \quad \text{and} \quad \{x\} \subseteq B^2.
\]
Hence,
\[
d(z, y) = \rho(z, y), \quad d(x, z) = d(x, y) = \max\{g(B^1), g(B^2)\}
\]
hold by \eqref{t4.5:e1}. Lemma~\ref{l4.3}, the inclusion \(\{y, z\} \subseteq B^1\), and the equality \(d(z, y) = \rho(z, y)\) imply the strict inequality
\begin{equation}\label{t4.5:e3}
d(z, y) < r_1.
\end{equation}
Now \eqref{t4.5:e2} follows from \eqref{t4.5:e3} and \(g(\mathcal{B}) \subset (r_1, 2r_1)\).

Let us prove the equality \(\tau_d = \tau_{\rho}\). The last equality holds iff, for every \((p_n)_{n \in \NN} \subseteq X\) and for every \(p \in X\), we have the equivalence
\begin{equation}\label{t4.5:e15}
\left(\lim_{n \to \infty} \rho(p_n, p) = 0\right) \Leftrightarrow \left(\lim_{n \to \infty} d(p_n, p) = 0\right).
\end{equation}

To prove \eqref{t4.5:e15}, we first note that \(\Lim_{n \to \infty} d(p_n, p) = 0\) implies \(\Lim_{n \to \infty} \rho(p_n, p) = 0\) because \(d(x, y) \geqslant \rho(x, y)\) for all \(x\), \(y \in X\) by \eqref{t4.5:e1}. Conversely, let \(\Lim_{n \to \infty} \rho(p_n, p) = 0\) hold. The family \(\mathcal{B}\) is a cover of \(X\). Hence, we can find \(B \in \mathcal{B}\) such that \(p \in B\). Since \(B\) is an open ball in \((X, \rho)\), the limit relation \(\Lim_{n \to \infty} \rho(p_n, p) = 0\) and the membership \(p \in B\) imply that \((p_n)_{n \in \NN}\) eventually belongs to \(B\). Thus, by \eqref{t4.5:e1}, we have \(\rho(p_n, p) = d(p_n, p)\) for all sufficiently large \(n\), that implies the equality \(\Lim_{n \to \infty} d(p_n, p) = 0\).

Let us consider now the case when the ultrametric space \((X, \rho)\) is totally bounded. By Corollary~\ref{c6.9}, there is a strictly decreasing sequence \((r_n)_{n \in \NN} \subseteq \RR^{+}\) such that \(\Lim_{n \to \infty} r_n = 0\) and 
\[
D(X, \rho) = \{r_n \colon n \in \NN\} \cup \{0\}.
\]
Since \((X, \rho)\) is totally bounded, we can find a finite set \(C^{1} \subseteq X\) such that
\begin{equation}\label{t4.5:e6}
X = \bigcup_{c \in C^{1}} B_{r_1}(c).
\end{equation}
Using Lemma~\ref{l4.6}, we may also suppose that \(B_{r_1}(c^1) \cap B_{r_1}(c^2) = \varnothing\) whenever \(c^1\) and \(c^2\) are different points of \(C^{1}\). Since \((X, \rho)\) is totally bounded but not compact, Proposition~\ref{p2.9} implies that there is a Cauchy sequence \((x_n)_{n \in \NN} \subseteq X\) which does not have any limit in \((X, \rho)\). Using the definition of Cauchy sequences we can find \(c_1 \in C^{1}\) such that \(x_n \in B_{r_1}(c_1)\) for all sufficiently large \(n \in \NN\). The ball \(B_{r_1}(c_1)\) is totally bounded as a subset of totally bounded space \((X, \rho)\). Consequently, there is a finite \(C^{2} \subseteq X\) satisfying
\begin{equation}\label{t4.5:e7}
B_{r_1}(c_1) \subseteq \bigcup_{c \in C^{2}} B_{r_2}(c).
\end{equation}
Without loss of generality we suppose that \(B_{r_1} (c) \cap B_{r_2} (c) \neq \varnothing\) holds for every \(c \in C^{2}\). Then, by Corollary~\ref{c4.4}, for every \(c \in C^{2}\) we have the inclusion \(B_{r_2} (c) \subseteq B_{r_1} (c_1)\) (because \(r_2 < r_1\)). Hence, \(C^{2} \subseteq B_{r_1} (c_1)\) holds and inclusion \eqref{t4.5:e7} comes to the equality
\begin{equation}\label{t4.5:e8}
B_{r_1}(c_1) = \bigcup_{c \in C^{2}} B_{r_2}(c).
\end{equation}
By Lemma~\ref{l4.6}, there is \(C^{2, 1} \subseteq C^{2}\) such that
\[
\bigcup_{c \in C^{2, 1}} B_{r_2}(c) = \bigcup_{c \in C^{2}} B_{r_2}(c)
\]
and
\begin{equation}\label{t4.5:e9}
B_{r_2}(c^1) \cap B_{r_2}(c^2) = \varnothing
\end{equation}
whenever \(c^1\) and \(c^2\) are distinct points of \(C^{2, 1}\). Replacing, if necessary, the set \(C^{2}\) by \(C^{2, 1}\), we can assume that \eqref{t4.5:e9} is satisfied whenever \(c^1\) and \(c^2\) are distinct points of \(C^{2}\). By definition of the Cauchy sequences, there is a ball \(B_{r_2}(x) \in \bB_{X, \rho}\) such that the sequence \((x_n)_{n \in \NN}\) eventually belongs to \(B_{r_2}(x)\). Since \((x_n)_{n \in \NN}\) also eventually belongs to \(B_{r_1}(c_1)\), we have
\begin{equation}\label{t4.5:e10}
B_{r_2}(x) \cap B_{r_1}(c_1) \neq \varnothing.
\end{equation}
By Corollary~\ref{c4.4}, the inequality \(r_2 < r_1\) and \eqref{t4.5:e10} imply the inclusion
\[
B_{r_2}(x) \subseteq B_{r_1}(c_1).
\]
Using \eqref{t4.5:e8} we also see that Corollary~\ref{c4.4} implies the equality \(B_{r_2}(x) = B_{r_2}(c_2)\) with an unique \(c_2 \in C^{2}\). 

Arguing similarly for every integer \(n \geqslant 3\), we can find a finite \(C^{n} \subseteq B_{r_{n-1}}(c_{n-1})\) and a point \(c_n \in C^{n}\) such that the sequence \((x_n)_{n \in \NN}\) eventually belongs to \(B_{r_{n}}(c_{n})\), and
\begin{equation}\label{t4.5:e11}
B_{r_{n-1}}(c_{n-1}) = \bigcup_{c \in C^{n}} B_{r_n}(c),
\end{equation}
and \(B_{r_n}(c^1) \cap B_{r_n}(c^2) = \varnothing\) whenever \(c^1\) and \(c^2\) are distinct points of \(C^{n}\).

Let us consider \(\mathcal{B} \subseteq \bB_{X, \rho}\) defined by the rule: \(B \in \mathcal{B}\) iff there are \(n \in \NN\) and \(c \in C^{n} \setminus \{c_n\}\) such that \(B = B_{r_n}(c)\), where \(\{c_n\}\) is the one point set consisting from the point \(c_n\). We claim that \(\mathcal{B}\) is a partition of the set \(X\). Indeed, it follows directly from the definition of the sets \(C^{n}\), \(n = 1\), \(2\), \(\ldots\), that every two \emph{different} \(B_1\), \(B_2 \in \mathcal{B}\) are disjoint. Moreover, every \(B \in \mathcal{B}\) is nonempty. Consequently, \(\mathcal{B}\) is a partition of \(X\) iff 
\begin{equation}\label{t4.5:e12}
X = \bigcup_{B \in \mathcal{B}} B
\end{equation}
holds. Suppose contrary that there is 
\[
p_0 \in X \setminus \left(\bigcup_{B \in \mathcal{B}} B\right).
\]
Then from \eqref{t4.5:e6} and 
\[
\bigcup_{c \in C^{1} \setminus \{c_1\}} B_{r_1}(c) \subseteq \bigcup_{B \in \mathcal{B}} B
\]
it follows that \(p_0 \in B_{r_1}(c_1)\). Similarly, using \eqref{t4.5:e8} we obtain that \(p_0 \in B_{r_2}(c_2)\). Using \eqref{t4.5:e11}, for arbitrary integer \(n \geqslant 3\), we can prove the membership 
\begin{equation}\label{t4.5:e13}
p_0 \in B_{r_n}(c_n)
\end{equation}
by induction. Since, for every \(n \in \NN\), the ball \(B_{r_n}(c_n)\) eventually contains the sequence \((x_n)_{n \in \NN}\), membership~\eqref{t4.5:e13} and the limit relation \(\Lim_{n \to \infty} r_n = 0\) imply the equality \(\Lim_{n \to \infty} \rho(x_n, p_0) = 0\). Thus, \((x_n)_{n \in \NN}\) converges to the point \(p_0 \in X\), contrary to the definition.

It should be also noted that the family \(\mathcal{B}\) is countably infinite. Indeed, \(\mathcal{B}\) is at most countable because \(C^{n}\) is finite for every \(n \in \NN\). If \(\mathcal{B}\) is finite, then there is \(n_0 \in \NN\) such that
\[
B_{r_{n_0}} (c_{n_0}) = B_{r_{n_0+1}} (c_{n_0+1}) = \ldots.
\]
Thus, \((x_n)_{n \in \NN}\) is eventually constant and, consequently, convergent, which is impossible as was shown above.

The equality \(|\mathcal{B}| = \aleph_0\) implies that there is a bijection \(g \colon \mathcal{B} \to (r_1, 2r_1) \cap \QQ\). Now we can define a mapping \(d \colon X \times X \to \RR^{+}\) by \eqref{t4.5:e1}. Since \(\mathcal{B}\) is a cover of \(X\), the mapping \((x, y) \mapsto d(x, y)\) is correctly defined. Reasoning as above, we can show that \(d \colon X \times X \to \RR^{+}\) is an ultrametric.

Now we can easily complete the proof of the truth of \(\ref{t4.5:s4} \Rightarrow \ref{t4.5:s1}\) for both completely bounded and not completely bounded \(\rho\). Indeed, \eqref{t4.5:e1} and the equality \(g(\mathcal{B}) = (r_1, 2r_1) \cap \QQ\) imply that 
\[
(r_1, 2r_1) \cap \QQ \subseteq D(X, d) \subseteq [0, 2r_1).
\]
Hence, we have
\[
2r_1 = \diam (X, d) \notin D(X, d).
\]
Consequently, the poset \(D(X, d)\) does not have the largest element (see Example~\ref{ex2.20}).
\end{proof}

\begin{corollary}\label{c4.6}
Let \((X, \tau)\) be an ultrametrizable infinite topological space. Then \((X, \tau)\) is compact if and only if every compatible with \(\tau\) ultrametric is totally bounded.
\end{corollary}

\begin{proof}
If \((X, \tau)\) is compact, then every compatible with \(\tau\) ultrametric is also compact and, consequently, totally bounded. 

Conversely, suppose every compatible with \(\tau\) ultrametric \(d \colon X \times X \to \RR^{+}\) is totally bounded. Then, by Theorem~\ref{t5.10}, all compatible with \(\tau\) ultrametrics have the order isomorphic distance sets. Hence, \((X, \tau)\) is compact by Theorem~\ref{t4.5}.
\end{proof}

\begin{remark}\label{r4.7}
It is interesting to compare Corollary~\ref{c4.6} with the following theorem: ``A metrizable topological space \((X, \tau)\) is compact if and only if every compatible with \(\tau\) metric is complete''. This result was proved by Niemytzki and Tychonoff in 1928 \cite{NTFM1928}. There is also an ultrametric modification of Niemytzki---Tychonoff theorem obtained by Yoshito Ishiki in 2020 \cite{Isha2020}.
\end{remark}

Using the ultrametric version Niemytzky--Tychonoff theorem \cite{Isha2020}, we can expand Corollary~\ref{c4.6} to the following result.

\begin{proposition}\label{p4.11}
Let \((X, \tau)\) be an ultrametrizable topological space. Then the following conditions are equivalent:
\begin{enumerate}
\item \label{p4.11:s1} \((X, \tau)\) is compact.
\item \label{p4.11:s2} Every compatible with \(\tau\) ultrametric is complete.
\item \label{p4.11:s3} Every compatible with \(\tau\) ultrametric is totally bounded.
\end{enumerate}
\end{proposition}

To prove Proposition~\ref{p4.11}, we only note that statements \ref{p4.11:s1}, \ref{p4.11:s2} and  \ref{p4.11:s3} are trivially true if \(X\) is a finite set.

Let us now turn to the ultrametrizable separable topological spaces.

\begin{lemma}\label{l4.10}
Let \((X, \rho)\) be an ultrametric space. Then the following conditions are equivalent:
\begin{enumerate}
\item \label{l4.10:s1} \((X, \rho)\) is not separable.
\item \label{l4.10:s2} There is \(r > 0\) such that
\[
X \setminus \bigcup_{a \in A} B_r(a) \neq \varnothing
\]
for every at most countable \(A \subseteq X\).
\end{enumerate}
\end{lemma}

\begin{proof}
It follows directly from the definition of separability that the implication
\[
\neg \ref{l4.10:s1} \Rightarrow \neg \ref{l4.10:s2}
\]
is true, where \(\neg\) is the negation symbol. Hence, it suffices to prove the validity of the implication \(\ref{l4.10:s1} \Rightarrow \ref{l4.10:s2}\). 

Suppose that \ref{l4.10:s1} holds but for every \(r > 0\) there exists an at most countable set \(A(r) \subset X\) such that
\[
X \subseteq \bigcup_{a \in A(r)} B_r(a).
\]
Write \(\QQ^{+} = (0, \infty) \cap \QQ\). Then the set 
\[
A = \bigcup_{r \in \QQ^{+}} A(r)
\]
is at most countable and dense in \((X, \rho)\). Hence, \((X, \rho)\) is separable, contrary to \ref{l4.10:s1}.
\end{proof}

\begin{theorem}\label{t4.8}
Let \((X, \tau)\) be an ultrametrizable topological space. Then the following conditions are equivalent:
\begin{enumerate}
\item \label{t4.8:s1} \((X, \tau)\) is separable.
\item \label{t4.8:s2} The distance set \(D(X, d)\) of an ultrametric space \((X, d)\) is at most countable whenever \(d \colon X \times X \to \RR^{+}\) is compatible with \(\tau\).
\end{enumerate}
\end{theorem}

\begin{proof}
\(\ref{t4.8:s1} \Rightarrow \ref{t4.8:s2}\). This implication is valid by Corollary~\ref{c2.10}. 

\(\ref{t4.8:s2} \Rightarrow \ref{t4.8:s1}\). Suppose \ref{t4.8:s2} holds but \((X, \tau)\) is not separable. Let us consider a compatible with \(\tau\) ultrametric \(\rho \colon X \times X \to \RR^{+}\). Then the ultrametric space \((X, \rho)\) is not separable. By Lemma~\ref{l4.10} there exists \(r_0 > 0\) such that the cover 
\[
\widetilde{\mathcal{B}} = \{B_{r_0}(c) \colon c \in X\}
\]
of \(X\) does not contain any countable subcover. Using Lemma~\ref{l4.6}, we can find \(C \subseteq X\) such that
\[
\mathcal{B} = \{B_{r_0}(c) \colon c \in C\}
\]
is a cover of \(X\) consisting disjoint balls only. Since \(\mathcal{B} \subseteq \widetilde{\mathcal{B}}\) holds, the set \(\mathcal{B}\) is uncountable. Hence, there is a mapping \(g \colon \mathcal{B} \to (r_0, 2r_0)\) having the uncountable range set,
\begin{equation}\label{t4.8:e1}
|g(\mathcal{B})| > \aleph_0.
\end{equation}
If we define \(d \colon X \times X \to \RR^{+}\) by formula~\eqref{t4.5:e1}, then, as was proved earlier, \(d\) is an ultrametric compatible with \(\tau\). Since the distance set \(D(X, d)\) contains the set \(g(\mathcal{B})\), inequality \eqref{t4.8:e1} implies that \(D(X, d)\) is uncountable, contrary to~\ref{t4.8:s2}.
\end{proof}

\section{Instead of conclusion}

The concept of ultrametrics admits a natural generalization for poset-valued mappings.

Let \((\Gamma, {\preccurlyeq})\) be a partially ordered set with the smallest element \(\gamma_0\) and let \(X\) be a nonempty set.

\begin{definition}\label{d5.1}
A mapping \(d_{\Gamma} \colon X^{2} \to \Gamma\)  is an \emph{ultrametric distance}, if the following conditions hold for all \(x\), \(y\), \(z \in X\) and \(\gamma \in \Gamma\).
\begin{enumerate}
\item \label{d5.1:s1} \(d_{\Gamma}(x, y) = \gamma_0\) if and only if \(x = y\).
\item \label{d5.1:s2} \(d_{\Gamma}(x, y) = d_{\Gamma}(y, x)\).
\item \label{d5.1:s3} If \(d_{\Gamma}(x, y) \preccurlyeq \gamma\) and \(d_{\Gamma}(y, z) \preccurlyeq \gamma\), then \(d_{\Gamma}(x, z) \preccurlyeq \gamma\).
\end{enumerate}
\end{definition}

The ultrametric distances were introduced by Priess-Crampe and Ribenboim \cite{PR1993AMSUH} and studied in~\cite{DovBBMSSS2020, PR1996AMSUH, PR1997AMSUH, Rib2009JoA, Rib1996PMH}. This generalization of ultrametrics has some interesting applications to logic programming, computational logic and domain theory \cite{Kro2006TCS, PR2000JLP, SH1998IMSB}.

For the case of totally ordered \((\Gamma, {\preccurlyeq})\), condition~\ref{d5.1:s3} from the above definition becomes the following modification of the strong triangle inequality
\begin{equation}
d_{\Gamma}(x, y) \preccurlyeq \max\left\{d_{\Gamma}(x, z), d_{\Gamma}(z, y)\right\}.
\end{equation}
Now we can define the ``open balls'' as the sets
\[
B_{\gamma}(c) = \{x \in X \colon d_{\Gamma}(c, x) \prec \gamma\},
\]
\(c \in X\), \(\gamma \in \Gamma \setminus \{\gamma_0\}\), and show that the set
\[
\bB_{X, d_{\Gamma}} = \{B_{\gamma}(c) \colon \gamma \in \Gamma \setminus \{\gamma_0\}, \ c \in X\}
\]
is a base for a topology \(\tau_{d_{\Gamma}}\) if \(|\Gamma| \geqslant 2\). It seems to be interesting to describe, for fixed  \((\Gamma, {\preccurlyeq})\), the topological spaces \((X, \tau)\) which are \(\Gamma\)-ultrametrizable in the sense that \(\tau = \tau_{d_{\Gamma}}\) holds with some ultrametric distance \(d_{\Gamma} \colon X^{2} \to \Gamma\).

\begin{conjecture}[prove or disprove]\label{c5.2}
Let \((\Gamma, {\preccurlyeq})\) be a totally ordered set and let \((X, \tau)\) be a \(\Gamma\)-ultrametrizable topological space. The following conditions are equivalent:
\begin{enumerate}
\item \label{c5.2:s1} \((\Gamma, {\preccurlyeq})\) is not a poset of order type \(\mathbf{1} + \bm{\omega}^{*}\) but there is a subposet \(\Gamma_0\) of \((\Gamma, {\preccurlyeq})\) such that \(\Gamma_0\) has the order type \(\mathbf{1} + \bm{\omega}^{*}\).
\item \label{c5.2:s2} \((X, \tau)\) is compact iff \(\{d_{\Gamma}(x, y) \colon x, y \in X\}\) has the largest element whenever \(\tau = \tau_{d_{\Gamma}}\).
\end{enumerate}
\end{conjecture}


\end{document}